\documentclass[12pt,reqno]{amsart}

\usepackage{amssymb}
\usepackage{amsmath,amsthm}
\usepackage{verbatim} 
\usepackage{xcolor}
\DeclareSymbolFont{bbold}{U}{bbold}{m}{n}
\DeclareSymbolFontAlphabet{\mathbbold}{bbold}

\usepackage{hyperref} 
\newcommand{\C}{\mathbb{C} }
\newcommand{\R}{\mathbb{R} }
\newcommand{\Z}{\mathbb{Z} }
\newcommand{\T}{\mathbb{T} }
\newcommand{\Q}{\mathbb{Q} }
\newcommand{\N}{\mathbb{N} }
\newcommand{\1}{\mathbbold{1} }

\newcommand{\inparentheses}[1]{\left( #1 \right)}

\newcommand{\eof}[1]{e \inparentheses{#1}}

\renewcommand{\vector}[1]{{\boldsymbol #1}}
\newcommand{\extend}{{E}}
\newcommand{\set}{{\mathcal{A}}}

\newtheorem{theorem}{Theorem}[section]
\newtheorem{lemma}[theorem]{Lemma}

\newtheorem{conjecture}[theorem]{Conjecture}

\theoremstyle{definition}

\theoremstyle{remark}

\def\dd{{\,{\rm d}}}

\theoremstyle{definition}

\theoremstyle{remark}

\numberwithin{equation}{section}

\begin{document}
\title[Discrete restriction for $(x,x^3)$]{Discrete restriction for $(x,x^3)$ and related 
topics}
\author[Kevin Hughes]{Kevin Hughes}
\address{School of Mathematics, University of Bristol, Fry Building, Woodland Road, 
Clifton, Bristol BS8 1UG, United Kingdom, and the Heilbronn Institute for Mathematical 
Research, Bristol, United Kingdom}
\email{khughes.math@gmail.com}
\author[Trevor D. Wooley]{Trevor D. Wooley}
\address{Department of Mathematics, Purdue University, 150 N. University Street, West 
Lafayette, IN 47907-2067, USA}
\email{twooley@purdue.edu}
\subjclass[2010]{42B05, 11L07, 42B37, 35Q53}
\keywords{Fourier series, discrete restriction estimates, KdV-like equations}
\dedicatory{Dedicated to the memory of Jean Bourgain}
\begin{abstract} Defining the truncated extension operator $E$ for a sequence $a(n)$ 
with $n \in \Z$ by putting
\[
\extend{a}(\alpha,\beta):=\sum_{|n|\le N}a(n) e(\alpha n^3 + \beta n),
\]
we obtain the conjectured tenth moment estimate
\[
\| \extend a \|_{L^{10}(\T^2)}\lesssim_\epsilon N^{\frac{1}{10}+\epsilon}
\|a\|_{\ell^2(\Z)}.
\]
We obtain related conclusions when the curve $(x,x^3)$ is replaced by $(\phi_1(x),
\phi_2(x))$ for suitably independent polynomials $\phi_1(x),\phi_2(x)$ having integer 
coefficients.
\end{abstract}
\maketitle
%
%

\section{Introduction}

We begin by recalling the discrete restriction conjecture for the curve $(x,x^3)$. 
Define the truncated extension operator $E$ for a sequence $a(n)$ with $n \in \Z$ by 
putting
\[
\extend{a}(\alpha,\beta):=\sum_{|n|\le N}a(n) e(\alpha n^3 + \beta n)
\]
for $\alpha,\beta \in \R$. Here and elsewhere, we write $\eof{t}$ in place of 
$e^{2\pi i t}$. Since $\eof{\cdot}$ is $\Z$-periodic, we may regard $\alpha$ and $\beta$ 
as elements of $\T:=\R/\Z$ or of any interval $I$ in $\R$ of length 1 without any 
confusion. Based on the usual heuristics in the circle method it is natural to make the 
following conjecture.
 
\begin{conjecture}\label{conjecture1.1}
For each $p \in [1,\infty]$ there exists a constant $C_p>0$ such that, for all $N \in \N$ 
and all sequences $a \in \ell^2(\Z)$, one has the discrete restriction bounds
\begin{equation}\label{eq:KdV}
\| \extend{a} \|_{L^{p}(\T^2)}\leq C_p \bigl( 1+N^{\frac{1}{2}-\frac{4}{p}}\bigr) 
\|a\|_{\ell^2(\Z)}.
\end{equation}

\end{conjecture}

Bourgain initiated the study of this restriction estimate in 
\cite{Bourgain:discrete_restriction:KdV}, wherein he proved the bound 
$\|Ea\|_{L^6(\T^2)} \lesssim_\epsilon N^\epsilon \|a\|_{\ell^2(\Z)}$ (see \cite[equation 
(8.37) on page 227]{Bourgain:discrete_restriction:KdV}). In order to facilitate further 
discussion we introduce a cruder version of the conjecture \eqref{eq:KdV}, to the effect 
that for each $\epsilon>0$, there exists a constant $C_{p,\epsilon}$ having the property 
that, for all $N \in \N$ and all sequences $a \in \ell^2(\Z)$, one has
\begin{equation}\label{eq:KdV:epsilon_loss}
\| \extend a \|_{L^{p}(\T^2)}\leq C_{p,\epsilon} N^\epsilon 
\left( 1+N^{\frac{1}{2}-\frac{4}{p}} \right) \|a\|_{\ell^2(\Z)}.
\end{equation}
In colloquial terms the estimate \eqref{eq:KdV:epsilon_loss} is the estimate 
\eqref{eq:KdV} with an ``$\epsilon$-loss''. Bourgain's work establishes this weaker 
conjecture for $1\le p\le 6$. The problem of proving Conjecture \ref{conjecture1.1} lay 
dormant for some time until Hu and Li \cite{HuLi:Degree3} established 
\eqref{eq:KdV:epsilon_loss} for $p=14$. We remark that Hu and Li conjectured 
\eqref{eq:KdV} for $2 \leq p \leq 8$ and \eqref{eq:KdV:epsilon_loss} for all 
$8\leq p\leq \infty$. Our conjecture here is a more optimistic version of \cite[equation 
(1.2)]{HuLi:Degree3} motivated by the observation that the underlying singular series 
does not diverge as it does in the quadratic case. Recently, Lai and Ding \cite{LaiDing} 
proved \eqref{eq:KdV:epsilon_loss} for $p=12$ using the recent resolution of the main 
conjecture in the discrete restriction analogue of the cubic case of Vinogradov's mean 
value theorem. The latter was noted first in \cite{Wooley:Restr} as a consequence of the 
methods of \cite{EC:cubic}, and was subsequently obtained by decoupling technology in 
\cite{BDG} and via efficient congruencing in \cite{EC:nested}.\par

Decoupling estimates and efficient congruencing estimates are stronger than discrete 
restriction estimates and therefore more difficult to obtain. By comparison with decoupling 
for the parabola, C. Demeter (personal communication) has shown that the analogous 
decoupling estimate for the curve $(x,x^3)$ fails in the range of exponents $6<p<12$. 
Despite this, in Section~\ref{section:cubiclinear} we obtain \eqref{eq:KdV:epsilon_loss} 
for $p=10$. 

\begin{theorem}\label{theorem:KdV10}
The estimate \eqref{eq:KdV:epsilon_loss} is true for $p=10$, and \eqref{eq:KdV} is true 
for all $p>10$. 
\end{theorem}

Our method of proof is motivated by corresponding techniques applied in the analogous 
number theoretic problem where $a(n)$ is identically $1$. In this situation (where 
$a(n)=1$ for all $n\in \Z$), the sixth moment estimate satisfies 
\[
\|a\|_{\ell^2(\Z)} \lesssim \|Ea\|_{L^6(\T^2)} \lesssim \|a\|_{\ell^2(\Z)} 
\]
and the ninth moment estimate satisfies 
\[
\|Ea\|_{L^9(\T^2)} \lesssim_\epsilon N^{\frac{1}{18}+\epsilon} \|a\|_{\ell^2(\Z)} 
\]
for all $\epsilon>0$; see \cite{Breteche,VW:nonary} and \cite{Wooley:odd} respectively. 
In Sections~\ref{section:2polys} and \ref{section:1poly} we extend our method to give 
new restriction estimates for related extension operators. Many of these estimates are not 
expected to be sharp.\par

In this paper we write $f(n) \lesssim g(n)$ to mean that there exists a constant $C>0$ 
with the property that $|f(n)|\leq Cg(n)$ for all $n$. This is equivalent to Vinogradov's 
notation $\ll$. Also, when $k\ge 2$, we write $\tau_k(n)$ for the $k$-fold 
divisor function defined via the relation
\[
\tau_k(n)=\sum_{\substack{d_1,\ldots ,d_k\in \N\\ d_1\ldots d_k=n}}1.
\] 

\subsection*{Acknowledgements}
The first author would like to acknowledge the Heilbronn Institute for Mathematical 
Research for its support. The second author's work was supported in part during the early 
phases of this research by a European Research Council Advanced Grant under the 
European Union's Horizon 2020 research and innovation programme via grant agreement 
No.~695223.

\section{The proof of Theorem~\ref{theorem:KdV10}}\label{section:cubiclinear}

It transpires that the full restriction estimate reported in Theorem~\ref{theorem:KdV10} 
is a consequence of the special case in which the sequence $a(n)$ is the characteristic 
function $\1_{\mathcal{A}}$ of a subset $\mathcal{A}$ of the truncated integers 
$\Z \cap [-N,N]$. We write $A$ for the cardinality of the set $\mathcal{A}$. Furthermore, 
in this context our extension operator is
\[
\extend\1_{\mathcal{A}}(\alpha,\beta) := \sum_{n \in \mathcal{A}} 
\eof{\alpha n^3 + \beta n}
\]
for $\alpha,\beta \in \T$. Our goal is the upper bound contained in the following theorem. 

\begin{theorem}\label{theorem:main}
There is a positive constant $\kappa$ such that, for each subset 
$\mathcal{A}\subset \Z \cap [-N,N]$ of cardinality $A$, one has 
\[
\int_{\T^2} \big| \extend\1_{\mathcal{A}}(\alpha,\beta) \big|^{10} \dd\alpha 
\dd\beta \lesssim N\exp \left( \kappa \frac{\log N}{\log\log N} \right) \cdot A^5.
\]
\end{theorem}

\begin{proof} Fix the interval $[-N,N]$ and subset $\set \subset \Z \cap [-N,N]$, and let 
$a=\1_{\set}$. The tenth moment $\|Ea\|_{10}^{10}$ counts the 
number of solutions to the system of equations 
\begin{align*}
\sum_{i=1}^3 (x_i^3-y_i^3)&=\sum_{i=4}^5 (x_i^3-y_i^3)\\
\sum_{i=1}^3 (x_i-y_i)&=\sum_{i=4}^5 (x_i-y_i), 
\end{align*}
with each $\vector{x},\vector{y} \in \set^5$. We will foliate over the possible common 
values 
\begin{equation}\label{kw1}
x_1-y_1+x_2-y_2+x_3-y_3=h=x_4-y_4+x_5-y_5,
\end{equation}
as $h$ varies over $\Z$. Since the set $\set$ is contained in $[-N,N]$, we find that 
solutions are possible only when $h \in [-4N,4N]$. Fourier analytically, we may then write 
$\|\extend a\|_{10}^{10}$ as
\[
\sum_{|h| \leq 4N} \int_\T \int_\T |\extend a(\alpha_1,\alpha_2)|^4 e(-\alpha_2h) \dd
\alpha_2 \int_\T |\extend a(\alpha_1,\alpha_3)|^{6} e(-\alpha_3h) \dd\alpha_3 \dd 
\alpha_1. 
\]
Taking absolute values and applying the triangle inequality we deduce that 
\begin{equation}\label{kw20}
\| \extend a \|_{10}^{10} \leq (8N+1) \int_\T \int_\T \int_\T 
|\extend a(\alpha_1,\alpha_2)|^4 \; |\extend a(\alpha_1,\alpha_3)|^{6}\dd \alpha_1 
\dd \alpha_2\dd \alpha_3 .
\end{equation}
Note here that we have thrown away potential oscillation in order to impose the restriction 
that $h=0$ in (\ref{kw1}).\par

We next foliate over common values in the cubic equation. When $t\in \N$ and 
$l\in \Z$, write $c_t(l)$ for the number of solutions of the simultaneous equations
\[
\sum_{i=1}^t(x_i^3-y_i^3)=l\quad \text{and}\quad \sum_{i=1}^t(x_i-y_i)=0,
\]
with $\vector{x},\vector{y}\in \mathcal{A}^t$. Then, in a manner similar to that 
underlying our earlier discussion regarding the linear equation, it follows via orthogonality 
that
\begin{equation}\label{kw2}
\| \extend a\|^{10}_{10}\le (8N+1)\sum_{|l|\le 4N^3}c_2(l)c_3(l).
\end{equation}
Our argument now divides into two parts according to whether the summand $l$ is zero 
or non-zero.\par

In order to treat the contribution in (\ref{kw2}) from the summand with $l=0$, we begin 
by observing that $c_2(0)$ counts the number of solutions of the simultaneous equations
$$x_1^3+x_2^3=y_1^3+y_2^3\quad \text{and}\quad x_1+x_2=y_1+y_2,$$
with $\vector{x},\vector{y}\in \mathcal A^2$. The contribution arising from those 
solutions with $x_1+x_2=0=y_1+y_2$ is plainly at most $A^2$. When $x_1+x_2\ne 0$, 
meanwhile, one may divide the respective left and right hand sides of these equations to 
deduce that $x_1^2-x_1x_2+x_2^2=y_1^2-y_1y_2+y_2^2$, whence $x_1x_2=y_1y_2$. 
Thus $\{x_1,x_2\}=\{y_1,y_2\}$, and there are at most $2A^2$ solutions of this type. 
We thus have $c_2(0)\le 3A^2$. Moreover, it is a consequence of the discussion 
surrounding \cite[equation (8.37)]{Bourgain:discrete_restriction:KdV} that for a suitable 
positive number $\kappa$, one has 
\begin{equation}\label{kw3}
c_3(0)\lesssim \exp\left( \kappa \log N/\log \log N\right)\cdot A^3.
\end{equation}
Since the argument of the latter source is more complicated than would be available via 
earlier methods (see \cite[Lemma 5.2 of Chapter V]{Hua1965}), and further fails 
to address the case $b=a^3$ of \cite[equation (8.37)]{Bourgain:discrete_restriction:KdV}, 
we presently make a detour to justify the estimate (\ref{kw3}). For now, it suffices to 
combine our estimates for $c_2(0)$ and $c_3(0)$ to obtain the bound
\begin{equation}\label{kw4}
c_2(0)c_3(0)\lesssim \exp (\kappa \log N/\log \log N)\cdot A^5.
\end{equation}

We now give an alternate argument to give the claimed bound on $c_3(0)$. 
Observe that $c_3(0)$ counts the number of solutions of the simultaneous equations
\begin{equation}\label{kw5}
x_1^3+x_2^3+x_3^3=y_1^3+y_2^3+y_3^3\quad \text{and}\quad 
x_1+x_2+x_3=y_1+y_2+y_3,
\end{equation}
with $\vector{x},\vector{y}\in \mathcal A^3$. Since
\[
(x_1+x_2+x_3)^3-(x_1^3+x_2^3+x_3^3)=3(x_1+x_2)(x_2+x_3)(x_3+x_1),
\]
we see that
\begin{equation}\label{kw6}
(x_1+x_2)(x_2+x_3)(x_3+x_1)=(y_1+y_2)(y_2+y_3)(y_3+y_1).
\end{equation}
Thus, in particular, if $x_i+x_j=0$ for some distinct indices $i$ and $j$ in $\{1,2,3\}$, then 
$y_{i'}+y_{j'}=0$ for some distinct indices $i'$ and $j'$ in $\{1,2,3\}$, and one has also 
$x_k=y_{k'}$ for some indices $k$ and $k'$ in $\{1,2,3\}$. In this way we see that 
there are $O(A^3)$ choices for $\vector{x}$ and $\vector{y}$ satisfying (\ref{kw5}) for 
which the left hand side of (\ref{kw6}) is $0$. Given any fixed one of the $O(A^3)$ 
choices for $\vector{x}\in \mathcal A^3$ in which the left hand side of (\ref{kw6}) is 
equal to a non-zero integer $L$, meanwhile, each factor on the right hand side of 
(\ref{kw6}) is equal to a divisor of $L$. It consequently follows that there are at most 
$8\max_{1\le n\le 8N^3}\tau_3(n)$ choices for (positive or negative) integers 
$d_1,d_2,d_3$ with $d_1d_2d_3=L$ having the property that
\[
y_1+y_2=d_1,\quad y_2+y_3=d_2,\quad y_3+y_1=d_3.
\]
Writing $M$ for the fixed integer $x_1+x_2+x_3$, we see that for a fixed choice of 
$\vector{d}$, one has
\[
y_1=M-d_2,\quad y_2=M-d_3,\quad y_3=M-d_1,
\]
so that $\vector{y}$ is also fixed. Making use of standard estimates for $\tau_3(n)$, we 
may thus conclude that there is a positive number $\kappa$ for which
\[
c_3(0)\lesssim A^3+A^3\max_{1\le n\le 8N^3}\tau_3(n)\lesssim 
\exp (\kappa \log N/\log \log N)\cdot 
A^3,
\]
justifying our earlier assertion.
\par

We next turn to consider the contribution in (\ref{kw2}) of the non-zero summands $l$. 
When $l$ is a fixed integer with $1\le |l|\le 4N^3$, we see that $c_2(l)$ is equal to the 
number of solutions of the simultaneous equations
\[
x_1^3+x_2^3-y_1^3-y_2^3=l\quad \text{and}\quad y_2=x_1+x_2-y_1.
\]
Substituting from the latter of these equations into the former, we obtain the equation
\[
(x_1+x_2-y_1)^3-(x_1^3+x_2^3-y_1^3)=-l,
\]
whence
\[
(x_1+x_2)(x_1-y_1)(x_2-y_1)=-l/3.
\]
We therefore deduce that $3|l$ and, as in the previous paragraph, there are at most 
$8\tau_3(|l/3|)$ possible choices for integers $e_1,e_2,e_3$ with $e_1e_2e_3=-l/3$ and
\[
x_1+x_2=e_1,\quad x_1-y_1=e_2,\quad x_2-y_1=e_3.
\]
For any fixed such choice of $\vector{e}$, one sees that
\[
e_1-e_2-e_3=2y_1,\quad e_2-e_3-e_1=-2x_2,\quad e_3-e_1-e_2=-2x_1,
\]
so that $x_1,x_2,y_1$ are fixed. Since $y_2=x_1+x_2-y_1$, it follows that $y_2$ is also 
fixed. Thus we have
\[
\max_{1\le |l|\le 4N^3}c_2(l)\lesssim \max_{1\le |n|\le 2N^3}\tau_3(n)\lesssim 
\exp \left( \kappa \log N/\log \log N\right) .
\]

Making use of our estimate for $c_2(l)$, we find that
\[
\sum_{1\le |l|\le 4N^3}c_2(l)c_3(l)\lesssim \exp \left( \kappa \log N/\log \log N\right) 
\sum_{|l|\le 6N^3}c_3(l).
\]
The last sum counts the number of solutions of the equation
\[
x_1+x_2+x_3=y_1+y_2+y_3,
\]
with $\vector{x},\vector{y}\in \mathcal A^3$, which is plainly $O(A^5)$. Thus we infer 
that
\[
\sum_{1\le |l|\le 4N^3}c_2(l)c_3(l)\lesssim \exp 
\left( \kappa \log N/\log \log N\right) \cdot A^5.
\]
The conclusion of the theorem follows by substituting this estimate and (\ref{kw4}) into 
(\ref{kw2}). 
\end{proof}

\begin{proof}[Proof of Theorem~\ref{theorem:KdV10}]
We now deduce Theorem~\ref{theorem:KdV10} from Theorem~\ref{theorem:main}. 
The argument to do so is a standard `vertical layer cake decomposition' argument in the 
theory of Lorentz spaces. Although an elementary dyadic decomposition argument suffices 
for our purposes, for the sake of concision it is expedient to make reference to 
\cite[Lemma~3.1]{GGPRY}. Thus, we recall the special case $p=2$ of the latter for the 
reader's convenience.
 
\begin{lemma}\label{lemma:basic_Lorentz}
Let $T : \C^N \to [0,\infty)$ be a sublinear function such that $T({\bf1}_\set) \leq C\|
{\bf1}_\set\|_{\ell^2}$ for all subsets $\set \in \C^N$. Then for all $a \in \C^N$, one has
\[
T(a) \leq 2^{1/2}C(2+(\log{N})^{1/2}) \|a\|_{\ell^2(\Z)}.
\]
\end{lemma}

We apply this lemma by taking $T(\cdot)$ to be $\|\extend{(\cdot)}\|_{10}$ with
\[
C=N \exp\left(\kappa \log{N}/\log\log{N}\right).
\]
Theorem~\ref{theorem:main} now implies that 
\[
\|\extend{a}\|_{10} 
\lesssim N \exp\left({\kappa\frac{\log{N}}{\log\log{N}}}\right) (1+(\log{N})^{1/2}) 
\|a\|_2.
\]
The conclusion of Theorem~\ref{theorem:KdV10} follows on noting that for all 
$\epsilon>0$, there exists a constant $C_\epsilon$ such that for all sufficiently large $N$, 
we have 
\[
\exp\left({\kappa\frac{\log{N}}{\log\log{N}}}\right) \left(1+(\log{N})^{1/2}\right) 
\leq C_\epsilon N^\epsilon.
\]

Our final task in the proof of Theorem~\ref{theorem:KdV10} is to prove \eqref{eq:KdV} 
for $p>10$. For this we use the ``$\epsilon$-removal lemmas'' \cite[Theorem~1.4 and 
Lemma~3.1]{HH:k-paraboloids}, which were adapted from 
\cite{Bourgain:discrete_restriction:NLS}. To be precise, in the statement of 
\cite[Lemma~3.1]{HH:k-paraboloids}, one takes $C=0$,  $p=10$, $q>10$ and 
$\zeta = 1/16$, and in the statement of \cite[Theorem~1.4]{HH:k-paraboloids}, one takes 
$d=1$ and $k=3$. 
\end{proof}

\section{Generalizations}\label{section:2polys}

We consider now the extension operator associated with two polynomials $\phi_1$ and 
$\phi_2$ with integral coefficients defined by
\[
\extend{a}(\alpha_1,\alpha_2):=\sum_{|n|\le N}a(n) 
e(\alpha_1\phi_1(n)+\alpha_2\phi_2(n))
\]
for $\alpha_1,\alpha_2 \in \R$. Since $e(\cdot)$ is $\Z$-periodic and the polynomials 
$\phi_1,\phi_2$ have integral coefficients, we may regard $\alpha_1$ and $\alpha_2$ as 
elements of $\T$ without any confusion. By making use of recent progress on decoupling 
and efficient congruencing, one may obtain the estimates contained in the following 
theorem. 

\begin{theorem}\label{theorem:1derivative}
Let $\phi_1,\phi_2$ be polynomials with integer coefficients and respective degrees $k_1$ 
and $k_2$ with $1 \leq k_1 \leq k_2$. If $\phi_1'$ and $\phi_2'$ are linearly independent 
over $\Q$, then we have
\begin{equation}\label{eq:1derivative:p=6}
\| \extend{a} \|_{L^{6}(\T^2)}
\lesssim_\epsilon
N^{\epsilon} \|a\|_{\ell^2(\Z)}
\end{equation}
and 
\begin{equation}\label{eq:1derivative:supercritical}
\| \extend{a} \|_{L^{{k_2(k_2+1)}}(\T^2)}
\lesssim_\epsilon
N^{\frac{1}{2}-\frac{k_1+k_2}{k_2(k_2+1)}+\epsilon} \|a\|_{\ell^2(\Z)}
\end{equation}
for each $\epsilon>0$ as $N \to \infty$.
\end{theorem}

\noindent 
Regarding estimate~\eqref{eq:1derivative:p=6}, see \cite[Corollary~1.3]{BD:Hypersurf} 
or the case $k=2$ and $s=3$ of \cite[Theorem~1.1]{EC:nested}. Meanwhile, when 
$\phi_1$ and $\phi_2$ are two distinct monomials, the 
estimate~\eqref{eq:1derivative:supercritical} is a special case of 
\cite[Theorem~1.1]{LaiDing}. The reader will have no difficulty in verifying that one may 
adapt the arguments of \cite{LaiDing} in a straightforward manner to handle the situation 
in which $\phi_1'$ and $\phi_2'$ are linearly independent over $\Q$. 

A further consequence of the efficient congruencing/decoupling machinery is the following 
theorem.
 
\begin{theorem}\label{theorem:2derivatives}
Let $\phi_1,\phi_2$ be polynomials with integer coefficients and respective degrees $k_1$ 
and $k_2$ with $\min\{k_1,k_2\} > 1$. If $\phi_1''$ and $\phi_2''$ are linearly 
independent over $\Q$, then we have the estimate 
\begin{equation}\label{eq:2derivatives}
\| \extend{a} \|_{L^{12}(\T^2)}
\lesssim_\epsilon
N^{\frac{1}{12}+\epsilon} \|a\|_{\ell^2(\Z)}
\end{equation}
for each $\epsilon>0$ as $N \to \infty$.
\end{theorem}

To derive this conclusion, one considers the auxiliary extension operator
\[
Fa(\alpha_1,\alpha_2,\alpha_3):=\sum_{|n|\le N}a(n)e(\alpha_1 
\phi_1(n)+\alpha_2\phi_2(n)+\alpha_3n),
\]
for $\alpha_1,\alpha_2,\alpha_3\in \R$. It is a consequence of the triangle inequality that
\[
\| \extend{a} \|_{L^{12}(\T^2)}
\leq (24N+1)^{\frac{1}{12}}\| Fa\|_{L^{12}(\T^3)}.
\]
Thus, the conclusion of Theorem \ref{theorem:2derivatives} follows from the estimate
\[
\| Fa \|_{L^{12}(\T^3)}\lesssim_\epsilon
N^{\epsilon} \|a\|_{\ell^2(\Z)}.
\]
This bound is immediate from the case $k=3$ and $s=6$ of 
\cite[Theorem 1.1]{EC:nested}, on checking that the Wronskian of first derivatives of the 
polynomials $\phi_1(t)$, $\phi_2(t)$ and $t$ is non-zero.\par

We expect the following sharp bound to hold in general.

\begin{conjecture}\label{conjecture:2polys}
Let $\phi_1,\phi_2$ be polynomials with integer coefficients having respective degrees 
$k_1$ and $k_2$ satisfying $\max\{k_1,k_2\} \geq 3$. If $\phi_1'$ and $\phi_2'$ are 
linearly independent over $\Q$, then for each $p \in [1,\infty]$, we have 
\[
\| \extend{a} \|_{L^{p}(\T^2)}
\lesssim 
\left( 1+N^{\frac{1}{2}-\frac{k_1+k_2}{p}} \right) \|a\|_{\ell^2(\Z)}
\]
as $N \to \infty$.
\end{conjecture}

\noindent
Note that the analogue of Conjecture \ref{conjecture:2polys} corresponding to the case 
$(k_1,k_2)=(1,2)$ cannot hold in the sharp form stated here, for an additional factor 
at least as large as $(\log N)^{1/6}$ is required when $p=6$ (see the discussion around 
\cite[equation (2.51)]{Bourgain:discrete_restriction:NLS}, wherein Bourgain obtained 
nearly optimal bounds for all $p \in [1,\infty]$). We will prove the following new bound 
towards this conjecture. 

\begin{theorem}\label{theorem:2polys}
Let $\phi_1,\phi_2$ be polynomials with integer coefficients and respective degrees $k_1$ 
and $k_2$ with $1 \leq k_1 < k_2$ and $k_2 \geq 3$. Then one has the estimate
\begin{equation}\label{eq:p=10}
\| \extend{a} \|_{L^{10}(\T^2)}\lesssim_\epsilon N^{\frac{1}{10}+\epsilon} 
\|a\|_{\ell^2(\Z)}
\end{equation}
for each $\epsilon>0$ as $N \to \infty$.
\end{theorem}

In situations in which $\phi_1$ is not linear, it follows by interpolating between the $6$-th 
moment estimate \eqref{eq:1derivative:p=6} and the $12$-th moment estimate 
\eqref{eq:2derivatives} that one has the bound
\[
\|\extend{a}\|_{L^{10}(\T^2)} \lesssim_\epsilon N^{\frac{1}{15}+\epsilon}
\|a\|_{\ell^2(\Z)}
\]
for all $\epsilon>0$. Consequently, in the proof below we may assume that $\phi_1$ is 
linear. Indeed it suffices to take $\phi_1(x)=x$.\par 

We will need a simple variant of \cite[Lemma 2]{Wsde} in the proof of 
Theorem~\ref{theorem:2polys}. We include a proof for the sake of completeness.

\begin{lemma}\label{lemma:sde}
Let $\psi(x_1,\dots ,x_s)$ be a non-zero multivariate polynomial with integer coefficients 
of total degree $k$. If $\set \subset \Z$ is a finite set of cardinality $A$, then the number 
of integer solutions to the equation $\psi(\vector{x})=0$ with $x_i \in \set$ for 
$i=1,\dots ,s$ is at most $kA^{s-1}$. 
\end{lemma}

\begin{proof} We proceed by induction on $s$. The desired conclusion plainly holds when 
$s=1$. Suppose that the conclusion of the lemma holds for each $s$ with $1\le s<t$, and 
let $\Psi \in \Z [x_1,\dots ,x_t]$ be a non-zero polynomial of total degree $k$. By 
rearranging variables, if necessary, we may suppose that $\Psi(x_1,\ldots,x_t)$ is a 
polynomial in $x_t$ with at least one non-zero coefficient. Let the degree of $\Psi $ with 
respect to $x_t$ be $r$, and suppose that the coefficient of $x_t^r$ is the polynomial 
$\Phi (x_1,\dots ,x_{t-1})$. Then $\Phi$ is a non-zero polynomial in $t-1$ variables of 
degree at most $k-r$. By the inductive hypothesis, the number of solutions of the equation 
$\Phi (x_1,\dots ,x_{t-1})=0$ with $x_i\in \set$ $(1\le i\le t-1)$ is at most $(k-r)A^{t-2}$. 
Then the number of solutions $(x_1,\dots ,x_t)$ of $\Psi (x_1,\ldots ,x_t)=0$ satisfying 
$\Phi (x_1,\dots ,x_{t-1})=0$ and with $x_i\in \set$ $(1\le i\le t)$ is at most 
$(k-r)A^{t-1}$. Meanwhile, if $\Phi (x_1,\dots ,x_{t-1})$ is non-zero then $x_t$ satisfies 
a non-trivial polynomial of degree $r$. So there are at most $rA^{t-1}$ solutions with 
$\Phi (x_1,\ldots ,x_{t-1})$ non-zero. We therefore conclude that there are at most 
$kA^{t-1}$ solutions altogether, and the inductive hypothesis holds with $t+1$ replacing 
$t$. This completes the proof of the lemma.
\end{proof}

\begin{proof}[Proof of Theorem~\ref{theorem:2polys}]
By the remark above we may assume that $\phi_1(x)=x$. As such, we write $\phi$ in 
place of $\phi_2$ and $k$ in place of $k_2$ in the proof. By 
Lemma~\ref{lemma:basic_Lorentz} we only need to prove \eqref{eq:p=10} for 
sequences $a$ which are the characteristic function of some subset 
$\set \subset \Z \cap [-N,N]$. Therefore, we want to bound the number of solutions to the 
system of equations 
\begin{align*}
\sum_{i=1}^3 (\phi(x_i)-\phi(y_i)) &= \sum_{i=4}^5 (\phi(x_i)-\phi(y_i)),\\
\sum_{i=1}^3 (x_i-y_i) &= \sum_{i=4}^5 (x_i-y_i), 
\end{align*}
with $\vector{x},\vector{y} \in \set^5$. As in the argument employed above to deliver 
the relation \eqref{kw20}, we find that at the expense of a factor of $8N+1$ we only 
need to bound the number of solutions to the system of equations 
\begin{equation}\label{kw21}
\begin{aligned}
\sum_{i=1}^3 (\phi(x_i)-\phi(y_i))&= \sum_{i=4}^5 (\phi(x_i)-\phi(y_i)),\\
\sum_{i=1}^3 (x_i-y_i) =\; &0 = \sum_{i=4}^5 (x_i-y_i), 
\end{aligned}
\end{equation}
with $\vector{x},\vector{y} \in \set^5$.\par

When $t\in \N$ and $l\in \Z$, we now write $c_t(l)$ for the number of solutions of the 
simultaneous equations
\[
\sum_{i=1}^t(\phi(x_i)-\phi(y_i))=l\quad \text{and}\quad \sum_{i=1}^t(x_i-y_i)=0,
\]
with $\vector{x},\vector{y}\in \mathcal{A}^t$. Then, by foliating over common values in 
the equation \eqref{kw21} involving $\phi$, just as in our proof of 
Theorem~\ref{theorem:main}, we find that a bound analogous to \eqref{kw2} holds in 
our present situation. That is, it follows via orthogonality that there exists a positive 
constant $C$ depending on the coefficients of $\phi$ such that 
\begin{equation}\label{eq:p=10:general}
\| \extend a\|^{10}_{10} 
\leq (8N+1) \sum_{|l|\le CN^k} c_2(l)c_3(l).
\end{equation}
Our argument again divides into two parts according to whether the summand $l$ is zero 
or non-zero.\par

In the present circumstances, one sees that $c_2(0)$ counts the number of solutions of 
the simultaneous equations 
\[
\phi(x_1)-\phi(y_1)+\phi(x_2)=\phi(y_2) 
\quad \text{and} \quad x_1-y_1+x_2=y_2.
\]
Upon substitution of the latter equation into the former, one finds that 
\[
\phi(x_1)-\phi(y_1)+\phi(x_2)-\phi(x_1-y_1+x_2) = 0.
\]
The polynomial on the left hand side has factors $x_1-y_1$ and $y_1-x_2$, whence there 
is a quotient polynomial $\psi(x_1,y_1,x_2)$ having integer coefficients with the property 
that 
\[
(x_1-y_1)(x_2-y_1)\psi(x_1,y_1,x_2)=0.
\]
The solutions with $x_1=y_1$ or $x_2=y_1$ contribute at most $2A^2$ solutions to the 
count $c_2(0)$. If, on the other hand, neither $x_1=y_1$ nor $y_1=x_2$, then 
$\psi(x_1,y_1,x_2)=0$. By Lemma \ref{lemma:sde}, the number of solutions of 
$\psi(x_1,y_1,x_2)=0$ with $x_1,y_1,x_2 \in \set$ is $O(A^2)$. Since $y_2$ is fixed by a 
choice for $x_1,y_1,x_2$, one infers that 
\begin{equation}\label{c20}
c_2(0) = O(A^2).
\end{equation} 
The estimate $c_3(0) \lesssim N^\epsilon A^3$ is immediate from 
\eqref{eq:1derivative:p=6}, and thus we conclude that 
\begin{equation}\label{kw9}
c_2(0)c_3(0) \lesssim N^\epsilon A^5.
\end{equation}

We turn next to the contribution in \eqref{eq:p=10:general} from the non-zero summands 
$l$. We begin by observing that $c_2(l)$ counts the number of solutions of the 
simultaneous equations 
\[
\phi(x_1)-\phi(y_1)+\phi(x_2)-\phi(y_2) = l
\quad \text{and} \quad x_1-y_1+x_2=y_2,
\]
with $x_1,y_1,x_2,y_2\in \set$.
As above, these equations imply that 
\[
(x_1-y_1)(x_2-y_1)\psi(x_1,y_1,x_2)=l.
\]
There are at most $8\tau_3(|l|)$ possible choices for non-zero integers $e_1,e_2,e_3$ 
with $e_1e_2e_3=l$, 
\begin{equation}\label{kw10a}
x_1-y_1=e_1, \quad x_2-y_1=e_2 
\quad \text{and} \quad 
\psi(x_1,y_1,x_2)=e_3. 
\end{equation}
For any fixed such choice of $\vector{e}$, one has $\psi(y_1+e_1,y_1,y_1+e_2)=e_3$. 
One has 
\[
-e_1e_2\psi(y_1+e_1,y_1,y_1+e_2)=\phi(y_1+e_1+e_2)-\phi(y_1+e_2)-\phi(y_1+e_1)+
\phi(y_1).
\]
The right hand side here is the second order difference polynomial associated with $\phi$, 
which is non-constant as a polynomial in $y_1$ because $\deg(\phi) = k \geq 3$. Thus the 
number of solutions for $y_1 \in \set$ to the equation $\psi(y_1+e_1,y_1,y_1+e_2)=e_3$ 
is $O(1)$. Any fixed choice of $y_1$ determines $x_1$ and $x_2$ via \eqref{kw10a}, and 
then $y_2=x_1-y_1+x_2$ is also determined. In this way we deduce that 
\begin{equation}\label{c2}
\max_{1 \leq |l| \leq CN^k} c_2(l) 
\lesssim \max_{1 \leq n \leq CN^k} \tau_3(n) 
\lesssim_\epsilon N^\epsilon. 
\end{equation}

Applying our newly obtained bound for $c_2(l)$ we find that 
\[
\sum_{1 \leq |l| \leq CN^k} c_2(l) c_3(l)
\lesssim_\epsilon N^\epsilon \sum_{1 \leq |l| \leq CN^k} c_3(l). 
\]
The last sum is bounded above by the number of solutions of the equation 
\[
x_1+x_2+x_3=y_1+y_2+y_3
\]
with $\vector{x},\vector{y} \in \set^3$, which is $O(A^5)$. 
Thus, 
\[
\sum_{1 \leq |l| \leq CN^k} c_2(l) c_3(l)
\lesssim_\epsilon N^\epsilon A^5,
\]
and we infer from \eqref{kw9} and \eqref{eq:p=10:general} that
\[
\|Ea\|_{10}^{10}\lesssim_\epsilon N^{1+\epsilon} A^5.
\]
The conclusion of the theorem now follows by invoking Lemma 
\ref{lemma:basic_Lorentz}. 
\end{proof}

\section{Discrete restriction for univariate polynomials}\label{section:1poly}

For $\phi$, a polynomial with integer coeffients of degree at least 3, we (re-)define our 
extension operator as 
\[
\extend a (\alpha) := \sum_{|n|\le N}a(n)e(\alpha \phi(n)),
\]
and we also make use of the auxiliary extension operator
\[
Fa(\alpha,\beta) := \sum_{|n|\le N}a(n)e(\alpha \phi(n)+\beta n) .
\]
These operators for a quadratic polynomial $\phi$ were studied by Bourgain in 
\cite{Bourgain:Squares}. The main goal of this section is the proof of the following 
theorem.

\begin{theorem}\label{theorem:cubes8}
Suppose that $\phi$ is a polynomial with integer coefficients of degree $k \geq 3$. 
For all $\epsilon>0$, there exists a constant $C_\epsilon>0$ such that 
\begin{equation}\label{eq:cubes4}
\| \extend a \|_{L^4(\T)}^4 \leq C_\epsilon N^{\epsilon} \|a\|_{\ell^2(\Z)}^4.
\end{equation}
and 
\begin{equation}\label{eq:cubes8}
\| \extend a \|_{L^8(\T)}^8 \leq C_\epsilon N^{1+\epsilon} \|a\|_{\ell^2(\Z)}^8.
\end{equation}
When $k=3$ and $p>8$, we have the sharp bound 
\begin{equation}\label{eq:cubes>8}
\| \extend a \|_{L^p(\T)} \lesssim_p N^{\frac{1}{2}-\frac{3}{p}} \|a\|_{\ell^2(\Z)}.
\end{equation}
\end{theorem}

When $\phi(n)$ has degree 3, the bound \eqref{eq:cubes8} is essentially sharp, up to the 
factor of $N^\epsilon$. Furthermore, when $a(n)$ is identically $1$, it follows from 
\cite[Theorem~2]{Vaughan} that there exists a positive constant $C$ such that 
\[
\| \extend a \|_{L^8(\T)}^8 \leq CN\|a\|_{\ell^2(\Z)}^8.
\]
Estimate \eqref{eq:cubes8} is not sharp in general. When $k \geq 27$, standard 
arguments lead from \cite[Theorem~1.1]{Marmon} to the conclusion that in the special 
case $\phi(n)=n^k$, there exists a positive constant $\delta$ depending on $k$ such that 
\[
\|Ea\|_{L^8(\T)}^8 \lesssim_\epsilon N^{1-\delta+\epsilon} \|a\|_{\ell^2(\Z)}^8
\]
for all $\epsilon>0$. Indeed, one may take
\[
1-\delta =\frac{16}{3\sqrt{3k}}+\max\left\{ \frac{2}{\sqrt{k}},
\frac{1}{\sqrt{k}}+\frac{6}{k+3}\right\}.
\]
Note that $1-\delta \rightarrow 0$ as $k\rightarrow \infty$. We expect the following sharp 
bound to hold in general.

\begin{conjecture}\label{conjecture:1poly}
Let $\phi$ be a polynomial with integer coefficients of degree $k \geq 3$. 
Then for each $p \in [1,\infty]$, we have 
\[
\| \extend{a} \|_{L^{p}(\T^2)}
\lesssim 
\left( 1+N^{\frac{1}{2}-\frac{k}{p}} \right) \|a\|_{\ell^2(\Z)},
\]
as $N \to \infty$.
\end{conjecture}

\begin{proof}[Proof of Theorem~\ref{theorem:cubes8}]
We begin with a proof of the fourth moment estimate \eqref{eq:cubes4}. Applying 
Lemma~\ref{lemma:basic_Lorentz}, we reduce to proving \eqref{eq:cubes4} for 
sequences given by the characteristic function of some subset of the integers. As such, fix 
$[-N,N]$ and our subset $\set \subset \Z \cap [-N,N]$. Let $a = \1_{\set}$. The fourth 
moment counts the number of solutions to the equation 
\begin{align*}
\phi(x_1)-\phi(y_1) = \phi(x_2)-\phi(y_2), 
\end{align*}
with $\vector{x},\vector{y} \in \set^2$. 
There exists a polynomial $\psi(x,y)$ with integer coefficients such that 
\[
\phi(x)-\phi(y) = (x-y)\psi(x,y).
\]
On writing $x=y+e$, one sees that 
\[
\phi(x)-\phi(y) = \phi(y+e)-\phi(y)
\]
is the first order difference polynomial associated with $\phi$. Since the degree of $\phi$ 
is at least 2, one has that $\psi(y+e,y)$ is not constant as a polynomial in $y$.\par 

We distinguish between two cases. The first case is when $\phi(x_1)-\phi(y_1) = 0$. In 
this case we have two further cases to consider: either $x_1=y_1$ or $\psi(x_1,y_1)=0$. 
By Lemma~\ref{lemma:sde}, there are at most $O(A)$ putative solutions of 
$\psi(x_1,y_1)=0$ with $x_1,y_1 \in \set$, and the same is self-evidently the case when 
$x_1=y_1$. It follows that there are at most $O(A)$ solutions to the equation 
$\phi(x_1)-\phi(y_1) = 0$. By symmetry, there are also at most $O(A)$ solutions to the 
equation $\phi(x_2)-\phi(y_2) = 0$. Hence these solutions contribute at most $O(A^2)$ 
solutions to the fourth moment. 

The second case is when $\phi(x_1)-\phi(y_1) \not= 0$. There are at most $A^2$ choices 
for $x_1,y_1$ in the set $\set$ with this property. Fixing any one such choice of 
$x_1,y_1$, we may assume that $\phi(x_1)-\phi(y_1)=l$ where $1 \leq |l| \leq CN^k$ for 
an appropriate constant $C$ depending on the coefficients of $\phi$. There are at most 
$4\tau_2(|l|)$ possible choices for non-zero integers $e_1,e_2$ with $e_1e_2=l$, 
\[
x_2-y_2=e_1 
\quad \text{and} \quad 
\psi(x_2,y_2)=e_2. 
\]
For any fixed choice of $e_1$ and $e_2$, one has $\psi(y_2+e_1,y_2)=e_2$. Since this 
polynomial equation is non-constant in $y_2$, there are at most $O(1)$ possible solutions 
for $y_2$. Consequently, there are at most $O(1)$ possible solutions for $x_2$. Thus, the 
contribution of the solutions of this second type to the fourth moment is 
\[
O \left( A^2\max_{1 \leq |l| \leq CN^k} \tau_2(|l|) \right) 
\lesssim_\epsilon N^\epsilon A^2.
\]
This completes the proof of the fourth moment estimate. 

We proceed now to examine the $8$-th moment. By applying 
Lemma~\ref{lemma:basic_Lorentz}, it suffices to prove \eqref{eq:cubes8} for sequences 
given by characteristic functions of subsets of the integers. With this observation in mind, 
we again fix $[-N,N]$ and our subset $\set \subset \Z \cap [-N,N]$. Also, let 
$a=\1_{\set}$. The eighth moment $\|Ea\|_8^{8}$ counts the number of solutions to the 
equation 
\begin{align*}
\sum_{i=1}^2 (\phi(x_i)-\phi(y_i)) = \sum_{i=3}^4 (\phi(x_i)-\phi(y_i)), 
\end{align*}
with each $\vector{x},\vector{y} \in \set^4$. We foliate our set of solutions over the 
solutions to the equation $h=x_1-y_2+x_2-y_2$ as $h$ ranges in $[-4N,4N]$. Writing this 
Fourier analytically, we thus deduce that
\[
\int_{\T} |\extend{a}({\alpha})|^{8} \dd {\alpha}= \sum_{|h|\le 4N} \int_{\T} \int_{\T} 
|Fa(\alpha,\beta)|^{4} |Ea(\alpha)|^{4} e(-\beta h) \dd \beta \dd \alpha.
\]
Taking absolute values, we may impose the restriction that $h=0$ and obtain the bound
\[
\| \extend a \|_8^{8}\leq (8N+1) \int_{\T} \int_{\T} |Fa(\alpha,\beta)|^{4} 
|Ea(\alpha)|^{4} \dd \beta \dd \alpha.
\]
The mean value on the right hand side here counts the number of solutions to the system 
of equations 
\begin{equation}
\begin{aligned}\label{kw30}
\sum_{i=1}^2 (\phi(x_i)-\phi(y_i)) &= \sum_{i=3}^4 (\phi(x_i)-\phi(y_i))\\
x_1-y_1 + x_2-y_2 &= 0,
\end{aligned}
\end{equation}
with $\vector{x},\vector{y}\in \set^4$.\par

Recall from Section~\ref{section:2polys} that when $l\in \Z$, we write 
$c_2(l)$ for the number of solutions of the simultaneous equations
\[
\phi(x_1)+\phi(x_2)-\phi(y_1)-\phi(y_2)=l\quad \text{and}\quad x_1-y_1+x_2=y_2,
\]
with $\vector{x},\vector{y}\in \mathcal{A}^2$. 
Also, when $l \in \Z$, write $c_2'(l)$ for the number of solutions of the equation
\[
\phi(x_1)+\phi(x_2)-\phi(y_1)-\phi(y_2)=l,
\]
with $\vector{x},\vector{y}\in \mathcal{A}^2$. By foliating over common values in the 
equation involving $\phi$ in \eqref{kw30}, we find that a bound analogous to \eqref{kw2} 
holds in our present situation. That is, 
\begin{equation*}
\| \extend a\|^{8}_{8} \leq (8N+1) \sum_{|l|\le CN^k} c_2(l) c_2'(l).
\end{equation*}
We have the trivial bound $\sum_{l \in \Z} c_2'(l) \leq A^4$ so that 
\[
\| \extend a\|^{8}_{8} \leq (8N+1) \Big( c_2(0)c_2'(0) + A^4 
\max_{1 \leq |l| \leq CN^k} c_2(l) \Big).
\]
Observe that, in view of the fourth moment estimate already derived, one has
\[
c'_2(0)=\|Ea\|_4^4\lesssim_\epsilon N^\epsilon A^2.
\]
Thus, on recalling also \eqref{c20} and \eqref{c2}, we deduce that 
\[
\| \extend a\|^{8}_{8} \lesssim_\epsilon N\big( N^\epsilon A^2 \cdot A^2 + 
A^4 N^\epsilon\big)\lesssim_\epsilon N^{1+\epsilon}A^4.
\]
From here, as we have already explained, the proof of the eighth moment estimate 
follows by appealing to Lemma \ref{lemma:basic_Lorentz}. 

Finally, by applying \cite[Theorem~4.1 and Lemma~3.1]{HH:k-paraboloids}, the estimate 
\eqref{eq:cubes>8} follows from \eqref{eq:cubes8} when is $\phi(n)=n^3$. The keen 
reader may verify that one may adapt the arguments of 
\cite[Section~4]{HH:k-paraboloids} to deduce \eqref{eq:cubes>8} for an arbitrary cubic 
polynomial having integer coefficients. To be precise, in the statement of 
\cite[Lemma~3.1]{HH:k-paraboloids}, one takes $C=0$,  $p=8$, $q>8$ and 
$\zeta = 2^{-3}$, and in the statement of \cite[Theorem~4.1]{HH:k-paraboloids}, one 
takes $\tau=1/4$. 
\end{proof}

%
%

%
%
\end{document}